%% file: main.tex
\documentclass[bibalpha]{amsart}
\usepackage{verbatim}
\usepackage{enumerate}
\usepackage{MnSymbol}
\usepackage[mathscr]{euscript}

\def \N{\mathbb{N}}

\def \sM{\mathscr{M}}

\def \sZ{\mathscr{Z}}

\def \Z{\mathbb{Z}}

\input{makros.tex}

\title{Dp-minimal expansions of discrete ordered abelian groups}

\author{Erik Walsberg}
\address
{Department of Mathematics\\University of Illinois at Urbana-Champaign\\1409 West Green Street\\Urbana, IL 61801}
\email{erikw@illinois.edu}
\urladdr{http://www.math.illinois.edu/\textasciitilde erikw}
\date{\today}

\begin{document}

\begin{abstract}
If $\sZ$ is a dp-minimal expansion of a discrete ordered abelian group $(Z,<,+)$ and $\sZ$ does not admit a nontrivial definable convex subgroup then $\sZ$ is interdefinable with $(Z,<,+)$ and $(Z,<,+)$ is elementarily equivalent to $(\Z,<,+)$.
\end{abstract}

\maketitle

\section{Introduction}
\noindent Suppose $T$ is a complete first order theory and $\sM$ is an $\aleph_0$-saturated model of $T$ with domain $M$.
Then $T$ is \textit{not} \textbf{dp-minimal} if there are formulas $\phi(x, \bar{y}), \varphi(x,\bar{z})$ and sequences $(\bar{a}_i)_{i \in \N}, (\bar{b}_i)_{i \in \N}$ of tuples from $\sM$ such that for all $i,j \in \N$ there is a $c \in M$ such that
$\sM \models \{ \phi(c,\bar{a}_i), \varphi(c,\bar{b}_j) \}$ and $\sM \models \{ \neg \phi(c,\bar{a}_m), \neg \varphi(c, \bar{b}_n)\}$ for all $m \neq i$, $n \neq j$.
A structure is dp-minimal if its theory is.
Dp-minimal theories are NIP theories and the theory of dp-minimal theories forms a well-behaved subtheory of the theory of NIP theories.
Many structures whose definable sets are ``tame topological" objects are dp-minimal.
Dp-minimal theories include o-minimal, C-minimal, P-minimal, strongly minimal theories, and Presburger arithmetic.

\medskip \noindent Two structures on a common domain $M$ are \textbf{interdefinable} if they define the same subsets of $M^n$ for all $n$.
An expansion of a structure $\sM$ is \textbf{proper} if it is not interdefinable with $\sM$.
The following is \cite[Proposition 6.6]{ADHMS}:

\begin{Thm}
\label{thm:adms}
There are no proper dp-minimal expansions of $(\Z,<,+)$.
\end{Thm}

\noindent This result generalizes to strong expansions of $(\Z,<,+)$ \cite[Corollary 2.20]{DG}.
We give another kind of generalization.

\medskip \noindent Throughout this paper $(Z,<,+)$ is a discrete ordered abelian group, $1_Z$ is the minimal positive element of $Z$, $O := \{ k1_Z : k \in \Z \}$, and $\sZ$ is an expansion of $(Z,<,+)$.
A subset $X$ of $Z$ is \textbf{convex} if whenever $x,x' \in X$ and $x < y < x'$ then $y \in X$.
A subgroup of $(Z,<,+)$ is non-trivial if it is not $Z$ or $\{0\}$.
Note $O$ is a convex subgroup of $(Z,<,+)$ and $(O,<,+)$ is canonically isomorphic to $(\Z,<,+)$.
A $\Z$-group is an ordered abelian group which is elementarily equivalent to $(\Z,<,+)$.

\begin{Thm}
\label{thm:main}
Suppose $\sZ$ is dp-minimal.
Then exactly one of the following holds:
\begin{enumerate}
    \item $\sZ$ defines a nontrivial convex subgroup of $(Z,<,+)$,
    \item $\sZ$ is interdefinable with $(Z,<,+)$ and $(Z,<,+)$ is a $\Z$-group.
\end{enumerate}
\end{Thm}

\noindent There are indeed dp-minimal expansions of discrete ordered abelian groups which define nontrival convex subgroups.
It is shown in \cite[Proposition 5.1]{JaSi-dp} that an ordered abelian group $\mathscr{G}$ is dp-minimal if and only if $|\mathscr{G}/n\mathscr{G}| < \aleph_0$ for all $n$.
Suppose $\mathscr{G}$ is a dp-minimal ordered abelian group.
Equip $\mathscr{G} \times \Z$ with the lexicographic order $<_{\text{Lex}}$.
Then $(\mathscr{G} \times \Z, <_{\text{Lex}})$ is a discrete ordered abelian group and $|(\mathscr{G} \times \Z)/n(\mathscr{G} \times \Z)| < \aleph_0$ for all $n$, so $(\mathscr{G} \times \Z,<_{\text{Lex}})$ is dp-minimal.
We now apply a corollary to Shelah's theorem on externally definable sets \cite{Shelah-strongly}, see \cite[Observation 3.8]{OnUs-dp}:

\begin{Fact}
\label{fact:externally}
Suppose $\sM = (M,<,\ldots)$ is a dp-minimal expansion of a linear order.
Then any expansion of $\sM$ by unary predicates defining convex subsets of $M$ is also dp-minimal.
\end{Fact}

\noindent It follows that any expansion of $\mathscr{G} \times \Z$ by convex subgroups is dp-minimal.

\medskip \noindent The proof of Theorem~\ref{thm:main} relies on previous work in the model theory of ordered abelian groups.
We now describe these results.

\medskip \noindent Michaux and Villemaire \cite{MiVi-Presburger} showed that an expansion $\sM$ of $(\Z,<,+)$ is interdefinable with $(\Z,<,+)$ if every $\sM$-definable subset of $\Z$ is $(\Z,<,+)$-definable.
This theorem was used in the proof of Theorem~\ref{thm:adms}.
We apply a generalization of this theorem due to Cluckers \cite[Theorem 5]{Cluckers}.
It not known if the assumption of $\aleph_0$-saturation in Theorem~\ref{thm:cluckers} is necessary.
This does not produce difficulties for us as we may pass to an elementary extension of $\sZ$ in the proof of Theorem~\ref{thm:main} without losing generality.

\begin{Thm}
\label{thm:cluckers}
Suppose $(Z,<,+)$ is a a $\Z$-group and suppose $\sZ$ is $\aleph_0$-saturated.
If every $\sZ$-definable subset of $Z$ is $(Z,<,+)$-definable then $\sZ$ and $(Z,<,+)$ are interdefinable.
\end{Thm}

\noindent Standard results from the model theory of ordered abelian groups yields the following:

\begin{Thm}
\label{thm:zgroup}
A discrete ordered abelian group is a $\Z$-group if and only if it does not admit a nontrivial definable convex subgroup.
\end{Thm}

\noindent We summarize the results which  imply Theorem~\ref{thm:zgroup} in the remainder of this introduction.

\medskip\noindent An ordered abelian group is $n$-regular if every open interval with cardinality at least $n$ contains an $n$-divisible element, an ordered abelian group is \textbf{regular} if it is $n$-regular for all $n \in \N$.
Robinson and Zakon \cite{Ro-Za} showed that an ordered abelian group is elementarily equivalent to an archimedean ordered abelian group if and only if it is regular.
Any discrete archimedean ordered abelian group is isomorphic to $(\Z,<,+)$, so it follows in particular that $(Z,<,+)$ is a $\Z$-group if and only if it is regular.
(This may already be deduced from the classical work of Presburger).

\medskip \noindent Suppose $(Z,<,+)$ is not regular.
Fix $n \in \N$ such that $(Z,<,+)$ is not $n$-regular.
Let $H$ be the set of $g \in Z$ such that every open interval of cardinality at least $n$ with endpoints in $[0,|g|]$ contains an $n$-divisible element.
Then $H$ is a convex subgroup, see \cite[Theorem 3.1]{Belegradek} for a proof.
As $(Z,<,+)$ is not $n$-regular, we see that $H \neq Z$, and $O \subseteq H$, so $H$ is nontrivial.

\section{Proof}
\noindent In this section we suppose that $\sZ$ is $\aleph_0$-saturated.

\begin{Lem}
\label{lem:max}
Suppose $\sZ$ does not admit a nontrivial definable convex subgroup.
Then every nonempty definable bounded above subset $Z$ has a maximum element.
\end{Lem}

\begin{proof}
Suppose $X$ is a nonempty definable bounded above subset of $X$.
Let
$$ Y := \{ y \in Z : (\exists x \in X)( y \leq x )\}. $$
It suffices to show that $Y$ has a maximal element.
Let
$$ H := \{ g \in Z : Y + g = Y  \}. $$
be the additive stabilizer of $Y$.
It is easy to see that $H$ is a convex subgroup of $(Z,<,+)$.
(This is true for any downwards closed subset of any ordered abelian group).
As $Y$ is bounded above $H$ cannot be all of $Z$, it follows that $H = \{0\}$.
Therefore $Y + 1_Z \neq Y$.
As $Y \subseteq Y + 1_Z$ we see that $Y + 1_Z$ is not a subset of $Y$.
Therefore there is an $h \in Y$ such that $h + 1_Z \notin Y$.
This $h$ is the maximal element of $Y$.
\end{proof}

\begin{Lem}
\label{lem:sparse}
Suppose $\sZ$ is dp-minimal and $A \subseteq Z$ is definable.
Suppose 
$$|A \cap (g + O)| \leq 1 \quad \text{for all} \quad g \in Z.$$
Then $A$ is finite.
\end{Lem}

\begin{proof}
We suppose towards a contradiction that $A$ is infinite.
Let $(a_i)_{i \in \N}$ be a sequence of distinct elements of $A$.
For every $i \neq j$ we have $|a_i - a_j| \geq h$ for each $h \in O$.
Saturation yields a positive $h \notin O$ such that $|a_i - a_j| \geq 2h$ when $i \neq j$.
Let
$$ \phi(x,y) := | x - y | < h $$
and 
$$ \varphi(x,y) := x - y \in A. $$
Consider the sequences $(a_i)_{i \in \N}$ and $(n1_Z)_{n \in \N}$.
Fix $j,m \in \N$.
Then 
$$\sZ \models \{ \phi(a_j + m1_Z, a_j), \varphi(a_j + m1_Z, m1_z) \}.$$
Suppose $i \neq j$ and $n \neq m$.
Then
$$ | a_i - (a_j + m1_Z) | = | (a_i - a_j) - m1_Z| \geq |a_i - a_j| - |m1_z| > h.  $$
Hence $\sZ \models \neg \phi(a_j + m1_Z, a_i)$.
Furthermore 
$$ (a_j + m1_Z) - n1_Z = a_j + (m - n)1_Z $$
is not an element of $A$ as $a_j$ is the only element of $A$ lying in $a_j + O$.
Hence $\sZ \models \neg \varphi( a_j + m1_Z, n1_Z)$.
So $\sZ$ is not dp-minimal.
\end{proof}

\noindent Proposition~\ref{prop:proof}, together with Theorems~\ref{thm:cluckers} and \ref{thm:zgroup}, yields Theorem~\ref{thm:main}.

\begin{Prop}
\label{prop:proof}
Suppose $\sZ$ is dp-minimal and does not admit a nontrivial definable convex subgroup.
Suppose $X \subseteq Z$ is definable.
Then $X$ is $(Z,<,+)$-definable.
\end{Prop}

\noindent We first make some definitions.
Let $A \subseteq \Z$.
Then $\Z$ is \textbf{periodic} (with period $n$) if for all $k \in \Z$ we have $k \in A$ if and only if $k + n \in A$.
A periodic subset of $\Z$ with period $n$ is a finite union of cosets of $n\Z$ and is hence $(\Z,+)$-definable.
We say that $A$ of $\Z$ is \textbf{eventually periodic} with period $n > 0$ if there are $m,m' \in \Z$ such that:
\begin{enumerate}
    \item if $k \geq m$ then $k \in A$ if and only if $k + n \in A$,
    \item if $k \leq m'$ then $k \in A$ if and only if $k - n \in A$.
\end{enumerate}
It is easy to see that any eventually periodic subset of $\Z$ is $(\Z,<,+)$-definable.
The quantifier elimination for $(\Z,<,+)$ shows that any $(\Z,<,+)$-definable subset of $\Z$ is eventually periodic.
If $A \subseteq \Z$ is eventually periodic with period $n$ then the minimal $m$ with the property above is (if it exists) the \textbf{upper period point} of $A$, the maximal $m'$ (if it exists) with the property above is the \textbf{lower period point} of $A$.
Note that $A$ does not have an upper period point if and only if $A$ does not have a lower period point if and only if $A$ is periodic.

\medskip \noindent Identifying $O$ with $\Z$ we also speak of periodic and eventually periodic subsets of $O$.
Note that $nO + g = [nZ + g] \cap O$ for all $n \in \N$ and $g \in O$.
It follows that if $A \subseteq O$ is eventually periodic then there is a $(Z,<,+)$-definable $B \subseteq Z$ such that $A = B \cap O$.

\medskip \noindent We now prove Proposition~\ref{prop:proof}.

\begin{proof}
Let $(\sZ,O)$ be the expansion of $\sZ$ by a unary predicate defining $O$.
Fact~\ref{fact:externally} shows $(\sZ,O)$ is dp-minimal.
It follows by Theorem~\ref{thm:adms} that any $(\sZ,O)$-definable subset of $O$ is $(O,<,+)$-definable.

Let $P_u$ be the set of $a \in X$ such that $0$ is the upper period point of $[X - a] \cap O$ and $P_\ell$ be the set of $a \in X$ such that $0$ is the lower period point of $[X - a] \cap O$.
Let
$$ \phi(g) := (\exists z \in O, z > 0) (\forall h \geq g)  [h - g \in O]  \Longrightarrow [ h \in X \Leftrightarrow h + z \in X]. $$
Then $P_u$ is defined by the $(\sZ,O)$-formula
$$ \phi(g) \land [(\forall h < g) [h - g \in O] \Longrightarrow \neg \phi(h)]. $$
Similar formulas may be used to see that $P_\ell$ is $(\sZ,O)$-definable.
Note that a coset of $O$ intersects $P_u$ if and only if it intersects $P_\ell$.
Note that $P_u$ contains at most one element from each coset of $O$, likewise for $P_\ell$.
Applying Lemma~\ref{lem:sparse} we see that $P_u$ and $P_\ell$ are finite.

Fix $g \in P_u$.
As $O \cap [X - g]$ is eventually periodic we suppose it has period $n$.
Then there is a finite union $Y$ of cosets of $nZ$ such that 
$$ \{ h \in O: h \geq 0 \} \cap Y = \{ h \in O : h \geq O \} \cap (X - g). $$
Let $C$ be the set of $h \in Z, h \geq g$ such that $[0,h-g] \cap Y = [0,h-g] \cap (X - g)$.
Suppose $C$ is bounded above.
Lemma~\ref{lem:max} shows that $C$ has a maximal element $h$.
It is easy to see that $h$ must be the minimal element of $\{ h' \in P_\ell : h' \geq h \}$.
In this case we have $[g,h] \cap (Y + g) = [g,h] \cap X$.
Suppose $C$ is not bounded above.
Then $[g,\infty) \cap (Y + g) = [g, \infty) \cap X$.
In this case $g$ must be the maximal element of $P_u$.

So we see that if $g \in P_u$ and $h$ is the minimal element of $\{ h' \in P_\ell : h' \geq g\}$ then $[g,h] \cap X$ is $(Z,<,+)$-definable and if $g$ is the maximal element of $P_u$ then $[g,\infty) \cap X$ is $(Z,<,+)$-definable.
A similar argument shows that if $g$ is the minimal element of $P_\ell$ then $(-\infty,g] \cap X$ is $(Z,<,+)$-definable.
It follows that $X$ is $(Z,<,+)$-definable.
\end{proof}

\bibliographystyle{alpha}
\bibliography{Ref}
\end{document}

%% file: makros.tex
\newtheorem{Th}{Theorem}[section]
\newtheorem{Thm}[Th]{Theorem}

\newtheorem{Prop}[Th]{Proposition}

\newtheorem{Lem}[Th]{Lemma}
\newtheorem{Fact}[Th]{Fact}